\documentclass[11pt]{amsart}   
\usepackage[english]{babel}
\usepackage{mathtools}
\usepackage{amssymb,amscd, amsthm,amsmath,epsfig, graphics,color,latexsym,cancel}   
\usepackage[linktoc=all, pagebackref, hyperindex]{hyperref}%
\hypersetup{colorlinks,  citecolor=blue,  filecolor=blue,  linkcolor=blue,  urlcolor=black}

\usepackage{tikz-cd} 
\usepackage{extpfeil}
\usepackage{tikz}
\usetikzlibrary{matrix,calc}
\usepackage{mathrsfs}
\usepackage[all]{xy}
\usepackage{amsfonts}
\usepackage[normalem]{ulem}
\usepackage{euscript}
\usepackage{amsmath}
\usepackage{enumitem}
\usepackage{ifpdf}
\usepackage{cleveref}
\usepackage{comment}
\LARGE\textwidth=6in
\newtheorem{Theorem}{Theorem}[section]
\newtheorem{Lemma}[Theorem]{Lemma}
\newtheorem{Corollary}[Theorem]{Corollary}
\newtheorem{Proposition}[Theorem]{Proposition}

\theoremstyle{definition}
\newtheorem{Remark}[Theorem]{Remark}
\newtheorem{Example}[Theorem]{Example}

\newtheorem{Definition}[Theorem]{Definition}

\def\sqr#1#2{{\vcenter{\hrule height.#2pt
			\hbox{\vrule width.#2pt height#1pt \kern#1pt
				\vrule width.#2pt}
			\hrule height.#2pt}}}
\def\phi{\varphi}

\def\VaVa{{\mathcal V}\kern-5pt {\mathcal V}}
\def\gr#1#2{{\rm gr}\, _{#1}(#2)}
\def\gr{{\rm gr}\,}

\def\depth{{\rm depth}\,}

\def\Min{{\rm Min}\,}
\def\codim{{\rm codim}\,}

\def\grade{{\rm grade}\,}
\def\rk{\rm rank}

\def\Ext#1#2#3#4{{\rm Ext}\,^{#1}_{#2}({#3},{#4})}

\def\supp#1{{\rm Supp}\, (#1)}

\def\ini{\mbox{\rm in}}

\def\cl#1{{\mathcal #1}}

\def\phi{\varphi}

\def\grade{{\rm grade}\,}

\def\fm{{\mathfrak m}}

\def\fp{{\mathfrak p}}

\def\fm{{\mathfrak m}}

\def\cl#1{{\cal #1}}
\def\rk{\rm rank}

\newcommand{\excise}[1]{}

\def\NZQ{\mathbb}               

\def\AA{{\NZQ A}}
\def\PP{{\NZQ P}}

\def\G{{\mathcal G}}

\def\opn#1#2{\def#1{\operatorname{#2}}} 
\opn\chara{char} \opn\length{\lambda} \opn\pd{pd} \opn\rk{rk}
\opn\projdim{proj\,dim} \opn\injdim{inj\,dim} \opn\rank{rank}
\opn\depth{depth} \opn\grade{grade} \opn\height{height}
\opn\embdim{emb\,dim} \opn\codim{codim}

\opn\Tr{Tr} \opn\bigrank{big\,rank}
\opn\superheight{superheight}\opn\lcm{lcm}
\opn\trdeg{tr\,deg}
	\opn\reg{reg} \opn\lreg{lreg} \opn\ini{in} \opn\lpd{lpd}
	\opn\size{size} \opn\sdepth{sdepth}
	\opn\link{link}\opn\fdepth{fdepth}\opn\lex{lex}
	\opn\tr{tr}
	\opn\type{type}
	\opn\div{div} \opn\Div{Div} \opn\cl{cl} \opn\Cl{Cl}
	\opn\Spec{Spec} \opn\Supp{Supp} \opn\supp{supp} \opn\Sing{Sing}
	\opn\Ass{Ass} \opn\Min{Min}\opn\Mon{Mon}
	\opn\Ho{H} \opn\der{Der}
	\opn\Ann{Ann} \opn\Rad{Rad} \opn\Soc{Soc}
	\opn\Im{Im} \opn\Ker{Ker} \opn\Coker{Coker} \opn\Am{Am}
	\opn\Hom{Hom} \opn\Tor{Tor} \opn\Ext{Ext} \opn\End{End}
	\opn\Aut{Aut} \opn\id{id}
	
	\opn\nat{nat}
	\opn\pff{pf}
	\opn\Pf{Pf} \opn\GL{GL} \opn\SL{SL} \opn\mod{mod} \opn\ord{ord}
	\opn\Gin{Gin} \opn\Hilb{Hilb}\opn\sort{sort}
	\opn\PF{PF}\opn\Ap{Ap} \opn\HS{HS}
	\opn\aff{aff} \opn
	\con{conv} \opn\relint{relint} \opn\st{st}
	\opn\lk{lk} \opn\cn{cn} \opn\core{core} \opn\vol{vol}  \opn\inp{inp} \opn\nilpot{nilpot}
	\opn\link{link} \opn\star{star}\opn\lex{lex}\opn\set{set}
	\opn\width{wd}
	\opn\Fr{F}
	\opn\QF{QF}
	\opn\G{G}
	\opn\type{type}\opn\res{res}
	\opn\log{Log}
	\opn\gr{gr}
	
	\def\pot#1#2{#1[\kern-0.28ex[#2]\kern-0.28ex]}

	\opn\dirlim{\underrightarrow{\lim}}
	\opn\inivlim{\underleftarrow{\lim}}
\begin{document}
		
		\title[Veronese Avoiding Hypersurfaces]{Veronese Avoiding Hypersurfaces}
	\author{Giovanna Ilardi}
		\address{Dipartimento Matematica Ed Applicazioni “R. Caccioppoli” Universit`a Degli Studi	Di Napoli “Federico II” Via Cintia - Complesso Universitario Di Monte S. Angelo 80126 - Napoli - Italia}
		\email{giovanna.ilardi@unina.it}
		\author{Abbas Nasrollah Nejad}
		\address{Department of Mathematics, Institute for Advanced Studies in Basic Sciences (IASBS), Zanjan 45137-66731, Iran}
		\email{abbasnn@iasbs.ac.ir}
		\author{Saeed Tafazolian}
		\address{Universidade Estadual de Campinas (UNICAMP) \\ Instituto de Matem\'atica, Estat\'{\i}stica e Computação Cient\'{\i}fica (IMECC) \\ Departamento de Matem\'atica \\
			Rua S\'ergio Buarque de Holanda, 651\\ 13083-970 Campinas-SP, Brazil}
		\email{saeed@unicamp.br}
\maketitle
\begin{abstract}
We introduce Veronese-Avoiding hypersurfaces, inspired by the theory of associated forms of Alper--Isaev. In the smooth case, we reinterpret their criterion via Macaulay inverse systems: the Veronese-Avoiding condition is equivalent to the non-degeneracy of the associated form. In the singular case, our main theorem shows that a reduced hypersurface with exactly $n$ isolated singular points is Veronese-Avoiding if and only if these points are ordinary nodes in general linear position; we also classify singular plane cubics and treat fewer than $n$ nodes via a natural rational map. We then study the parameter space, proving local closedness and identifying a distinguished irreducible nodal locus. Finally, we prove a Lefschetz-type consequence for the 
Milnor algebra in degree $1$.
\end{abstract}
\section{Introduction}

Let $R=k[x_1,\ldots,x_n]$ be a standard graded polynomial ring over an
algebraically closed field $k$ of characteristic zero, and let $f\in R_d$ be a
reduced homogeneous polynomial of degree $d\ge 3$. We denote by
$X=V(f)\subseteq \PP_k^{\,n-1}$ the corresponding projective hypersurface. The
gradient ideal of $f$ is
\[
J_f=\left(\frac{\partial f}{\partial x_1},\ldots,
\frac{\partial f}{\partial x_n}\right),
\]
and the quotient $M_f=R/J_f$ is the Milnor algebra of $f$. Since
$\operatorname{char}(k)=0$, Euler's relation gives $f\in J_f$, so the gradient
ideal defines the singular subscheme of $X$.

The purpose of this paper is to study a geometric condition on the top graded
piece of the Jacobian ideal. Set $T=n(d-2)$, and let
$V_n^m\subseteq \PP(R_m)$ denote the degree-$m$ Veronese variety, that is, the
image of the map $[\ell]\mapsto [\ell^m]$. We say that $X$ is
\emph{Veronese-Avoiding} if
\[
\codim_{\PP(R_{T-1})}\PP((J_f)_{T-1})=n
\quad\text{and}\quad
\PP((J_f)_{T-1})\cap V_n^{T-1}=\emptyset.
\]
The first condition will be called the \emph{gradient-generic condition}, while
the second is the \emph{Veronese-avoidance condition}.
Equivalently, the first condition says that $\dim_k(M_f)_{T-1}=n$, while the second says that no nonzero power $\ell^{T-1}$ belongs to $(J_f)_{T-1}$. Thus
the definition combines a numerical condition on the Milnor algebra with a
geometric avoidance condition for a linear subspace in $\PP(R_{T-1})$.

This condition is motivated by the interaction between Veronese varieties,
Macaulay inverse systems, Lefschetz properties, and associated forms. In
\cite{Alex-Giovanna-Abbas}, the avoidance of a Veronese variety by a linear
subspace is used to prove a Strong Lefschetz Property in degree $1$ for
zero-dimensional complete intersections. On the other hand, in the work of
Eastwood--Isaev and Alper--Isaev on associated forms and hypersurface
singularities, the same condition appears in a different language; see
\cite{Alper-Isaev1, Alper-Isaev2, Eastwood-Isaev}. If $f$ is smooth, then $J_f$ is a complete
intersection, $M_f$ is an Artinian Gorenstein algebra of socle degree $T$, and
the Macaulay inverse system of $M_f$ is the associated form of $f$. In this
setting, Alper--Isaev show that the associated form is non-degenerate precisely
when
\[
(J_f)_{T-1}\cap \{L^{T-1}:L\in R_1\}=0.
\]
In projective terms, this is exactly the Veronese-Avoiding condition. Thus the
smooth case connects our definition with classical invariant theory, Artinian
Gorenstein algebras, and the reconstruction problem for isolated hypersurface
singularities.

We first revisit the smooth case from the point of view of Macaulay inverse
systems. If $f$ defines a smooth hypersurface and $F$ is the Macaulay inverse
system of $M_f$, then Theorem~\ref{nonsingularAssform} shows that $V(f)$ is
Veronese-Avoiding if and only if $V(F)$ is smooth. This gives a direct
inverse-system interpretation of the Alper--Isaev criterion and provides a
practical way to test the Veronese-Avoiding condition in the complete
intersection case.

The main part of the paper concerns singular hypersurfaces. For hypersurfaces
with isolated singularities, the condition
$\codim \PP((J_f)_{T-1})=n$ is related to the defect of the singular subscheme
and to the coincidence threshold of the Milnor algebra; see
Lemma~\ref{lem:HFMilnor}. This allows us to connect the Veronese-Avoiding
property with the geometry of the singular locus. The central result is
Theorem~\ref{thm:main-singular}: if $X$ has exactly $n$ isolated singular
points, then $X$ is Veronese-Avoiding if and only if these points are ordinary
nodes in general linear position.

The proof reflects the two parts of the definition. If $X$ is
Veronese-Avoiding, the codimension condition forces the total Tjurina number to
be $n$. Since there are exactly $n$ singular points, each local Tjurina number
is equal to one, and hence each singularity is an ordinary node. The same
numerical condition also forces the points to impose independent linear
conditions on $R_1$, which is equivalent to general linear position. Conversely,
if the singular locus consists of $n$ nodes in general linear position, then the
self-duality of the Jacobian module shows that $(J_f)_{T-1}$ coincides with the
degree-$(T-1)$ part of the ideal of these points; in this form the avoidance of
the Veronese variety is immediate.

This classification also explains a construction used by Alper--Isaev in the
proof of \cite[Proposition~4.3]{Alper-Isaev1}. They use the auxiliary form
\[
f_0=
\begin{cases}
\displaystyle \sum_{1\leq i<j<k\leq n}x_ix_jx_k, & d=3,\\[6pt]
\displaystyle \sum_{1\leq i<j\leq n}
\left(x_i^{d-2}x_j^2+x_i^2x_j^{d-2}\right), & d\geq 4.
\end{cases}
\]
They prove that $(J_{f_0})_{T-1}$ has codimension $n$ and avoids the cone of
powers $\{L^{T-1}:L\in R_1\}$. In our terminology, this says that $f_0$ is
Veronese-Avoiding. The hypersurface $V(f_0)$ has precisely the coordinate
points as ordinary nodes, and these points are in general linear position.
Thus Theorem~\ref{thm:main-singular} gives a conceptual explanation of the
Alper--Isaev construction: their auxiliary form lies on the nodal boundary of
the smooth Veronese-Avoiding locus.

We also treat two related singular situations. First, for reduced singular plane
cubics, Proposition~\ref{prop:cubic-plane-classification} gives a complete
classification: such a cubic is Veronese-Avoiding if and only if it is nodal.
This includes the irreducible nodal cubic, the union of a smooth conic and a
transverse line, and the union of three distinct lines in general position.

Second, we consider what happens when the singular locus has fewer than $n$
points. Here a new phenomenon appears, emphasized in
Remark~\ref{rem:global-phenomenon}: the Veronese-Avoiding property is no longer
determined by the number of singular points or by their local analytic type.
For instance, the plane curves
\[
f_d=xyz^{d-2}+x^d+y^d
\quad\text{and}\quad
g_d=xyz^{d-2}+x^d+y^d+x^{d-1}z,
\qquad d\ge 4,
\]
both have a unique singular point at $[0:0:1]$, and in both cases this
singularity is a node. Nevertheless, $V(f_d)$ is not Veronese-Avoiding, whereas
$V(g_d)$ is Veronese-Avoiding. Thus, outside the range covered by
Theorem~\ref{thm:main-singular}, the Veronese-Avoiding condition is a genuinely
global property of the Jacobian ideal, not a condition depending only on the
local singularity type.

This phenomenon is explained by Theorem~\ref{thm:nodal-r-points}. If
$\Sing(X)=\Gamma$ consists of $r<n$ nodes imposing independent linear conditions
on $R_1$, then the gradient-generic condition
holds automatically, and the essential issue is the Veronese-avoidance condition. Theorem~\ref{thm:nodal-r-points} shows that this avoidance condition is
controlled by the base locus of the rational map
\[
\phi_f:\PP(I(\Gamma)_1)\dashrightarrow \PP(N(f)_{T-1}),\qquad
[\ell]\longmapsto [\ell^{T-1}\bmod (J_f)_{T-1}],
\]
where $N(f)=J_f^{\mathrm{sat}}/J_f$ is the Jacobian module. Since both vector
spaces have dimension $n-r$, this may be viewed as a rational map
\[
\phi_f:\PP^{\,n-r-1}\dashrightarrow \PP^{\,n-r-1}.
\]
Thus, under these hypotheses, $X$ is Veronese-Avoiding if and only if this map
has no base points.

We next study the parameter space of Veronese-Avoiding hypersurfaces. Let
\[
\mathscr V_{n,d}=
\{\, [f]\in \PP(R_d)\mid V(f)\subseteq \PP_k^{n-1}
\text{ is Veronese-Avoiding}\,\}.
\]
Using the map
\[
\mu_f:R_{T-d}^{\oplus n}\to R_{T-1},\qquad
(a_1,\ldots,a_n)\mapsto \sum_{i=1}^n a_i\frac{\partial f}{\partial x_i},
\]
we realize the codimension condition as a constant-rank condition, and the
avoidance condition as the inverse image of an open subset of a Grassmannian.
It follows from Corollary~\ref{cor:locally-closed-VAH} that
$\mathscr V_{n,d}$ is locally closed in $\PP(R_d)$. We then construct a
distinguished nodal locus: by Proposition~\ref{prop:nodal-locus-irred}, the
locus $\mathscr N_{n,d}$ of hypersurfaces with exactly $n$ nodes in general
linear position is nonempty, irreducible, constructible, and has dimension
\[
\binom{n+d-1}{d}-1-n.
\]
Theorem~\ref{thm:nodal-stratum-in-VAH} shows that
$\mathscr N_{n,d}\subseteq \mathscr V_{n,d}^{\mathrm{red}}$, and
Corollary~\ref{cor:smooth-VAH-nonempty} gives nonemptiness of the smooth
Veronese-Avoiding locus.

Finally, we prove a Lefschetz-type consequence of the definition. Although this
consequence follows from \cite[Theorem~1.5 and Remark~1.7]{Alex-Giovanna-Abbas},
we give a direct proof in the present setting. Theorem~\ref{thm:VAH-degree-one-Lefschetz}
shows that if $X$ is Veronese-Avoiding, then for a general linear form $\ell$
the multiplication map
\[
\ell^{T-2}:(M_f)_1\longrightarrow (M_f)_{T-1}
\]
is an isomorphism. Geometrically, the proof comes from projecting the Veronese
variety $V_n^{T-1}$ away from the linear space $\PP((J_f)_{T-1})$; the
avoidance condition guarantees that the projection is defined everywhere on the
Veronese variety. In the case of nodal singular plane cubics, this specializes
to the ordinary weak Lefschetz property in degree $2$.

\subsection*{Acknowledgments}
The second author was partially supported by the Iran National Science Foundation (INSF) under project No.~4046270.
The third author was partially supported by CNPq grant no. 302774/2025-4, FAEPEX grant no. 3485/25, and FAPESP grant no. 2024/00923-6.

\section{Veronese-avoiding  hypersurfaces}\label{sectionVAH}
Let $R=k[x_1,\ldots,x_n]$ be a standard graded polynomial ring over an algebraically closed field $k$ of characteristic zero, and let $f\in R_d$ be a reduced homogeneous polynomial of degree $d\geq 3$. We denote by $X=V(f)\subseteq \PP_k^{n-1}$ the corresponding reduced hypersurface. The gradient ideal of $f$ is
\[
J_f=\left(\frac{\partial f}{\partial x_1},\ldots,\frac{\partial f}{\partial x_n}\right),
\]
and the quotient $M_f:=R/J_f$ is called the Milnor algebra of $f$.

For an integer $m\geq 1$, consider the degree-$m$ Veronese embedding
\[
v_n^m:\PP(R_1)\longrightarrow \PP(R_m),\qquad [\ell]\longmapsto [\ell^m].
\]
Its image will be denoted by
\[
V_n^m:=v_n^m(\PP(R_1))\subseteq \PP(R_m).
\]
For $m\geq 2$, if one writes homogeneous coordinates on $\PP(R_m)$ indexed by the monomials of degree $m$, then $V_n^m$ is described by the rank-one condition on the first catalecticant matrix $\mathbf M_{n,m}$. Equivalently, its defining ideal is generated by the $2\times 2$ minors of $\mathbf M_{n,m}$; see \cite[Theorem 4]{Barshay}.

Set $T=n(d-2)$. Since $(J_f)_{T-1}\subseteq R_{T-1}$ is a $k$-linear subspace, its projectivization
\[
\PP((J_f)_{T-1})\subseteq \PP(R_{T-1})
\]
is a projective linear subspace whenever $(J_f)_{T-1}\neq 0$. Moreover,
\[
\codim_{\PP(R_{T-1})}\PP((J_f)_{T-1})
=\dim_k R_{T-1}-\dim_k (J_f)_{T-1}
=\dim_k(M_f)_{T-1}.
\]

\begin{Definition}\label{VAH}
Let $X=V(f)\subseteq \PP_k^{n-1}$ be a reduced hypersurface of degree $d\geq 3$. We say that $X$ is \emph{Veronese-Avoiding} if the following two conditions hold:
\begin{enumerate}
\item[{\rm (I)}] {\rm (Gradient-generic condition)}
\[
\codim_{\PP(R_{T-1})}\PP((J_f)_{T-1})=n.
\]

\item[{\rm (II)}] {\rm (Veronese-avoidance condition)}
\[
\PP((J_f)_{T-1})\bigcap V_n^{T-1}=\emptyset.
\]
\end{enumerate}
\end{Definition}
Before turning to the general theory, we record two elementary examples: the first shows that the gradient-generic condition alone does not imply the Veronese-Avoiding property, while the second provides a singular Veronese-Avoiding curve.

\begin{Example}\label{ex:Fermat}
Let $f=x_1^d+\cdots+x_n^d$. Then $J_f=(x_1^{d-1},\ldots,x_n^{d-1})$, and
\[
M_f=R/J_f \simeq k[x_1,\ldots,x_n]/(x_1^{d-1},\ldots,x_n^{d-1})
\]
is a monomial complete intersection with socle degree $T=n(d-2)$. The monomials of degree $T-1$ surviving in $M_f$ are precisely
\[
x_1^{d-2}\cdots x_i^{d-3}\cdots x_n^{d-2},\qquad i=1,\ldots,n.
\]
Hence $\dim_k(M_f)_{T-1}=n$, so condition {\rm (I)} holds. On the other hand, since $T-1\geq d-1$, we have
\[
x_1^{T-1}=x_1^{T-d}x_1^{d-1}\in (J_f)_{T-1}.
\]
Thus $[x_1^{T-1}]\in \PP((J_f)_{T-1})\bigcap V_n^{T-1}$, and condition {\rm (II)} fails. Therefore the Fermat hypersurface satisfies condition {\rm (I)} but is not Veronese-Avoiding.
\end{Example}
\begin{Proposition}\label{prop:family-VAH-iff-cubic}
For each integer $d\geq 3$, let $f_d=xyz^{d-2}+x^d+y^d\in R=k[x,y,z]$. Then $X_d=V(f_d)\subseteq \PP_k^2$ is Veronese-Avoiding if and only if $d=3$.
\end{Proposition}

\begin{proof}
Assume first that $d=3$. Then $f_3=xyz+x^3+y^3$ and
\[
J_{f_3}=(3x^2+yz,\;3y^2+xz,\;xy).
\]
Since $T-1=3(d-2)-1=2$, one has
\[
(J_{f_3})_2=\langle 3x^2+yz,\;3y^2+xz,\;xy\rangle_k.
\]
With respect to the ordered basis $x^2,\;xy,\;xz,\;y^2,\;yz,\;z^2$
of $R_2$, the corresponding coefficient matrix is
\[
N=
\begin{bmatrix}
3&0&0&0&1&0\\
0&0&1&3&0&0\\
0&1&0&0&0&0
\end{bmatrix}.
\]
Hence the linear space $L:=\PP((J_{f_3})_2)\subseteq \PP(R_2)=\PP_k^5$
is defined by
\[
z_1-3z_5=0,\qquad z_4-3z_3=0,\qquad z_6=0.
\]
Therefore $\codim L=3$, so condition {\rm (I)} holds.

Now $V_3^2\subseteq \PP_k^5$ is defined by the $2\times 2$ minors of the symmetric matrix 
\[
M=
\begin{bmatrix}
z_1&z_2&z_3\\
z_2&z_4&z_5\\
z_3&z_5&z_6
\end{bmatrix}.
\]
Restricting to $L$, we obtain
\[
\widetilde M=
\begin{bmatrix}
3z_5&z_2&z_3\\
z_2&3z_3&z_5\\
z_3&z_5&0
\end{bmatrix}.
\]
A direct computation shows that the $2\times 2$ minors of $\widetilde M$ generate $(z_2,z_3,z_5)^2$. 
Hence the ideal generated by the equations of $L$ together with the defining equations of $V_3^2$ is $\fm$-primary, where $\fm=(z_1,\ldots,z_6)$. Thus $L\bigcap V_3^2=\emptyset$ and condition {\rm (II)} also holds. Therefore $X_3$ is Veronese-Avoiding.

Assume now that $d\geq 4$. Then
\[
\frac{\partial f_d}{\partial x}=yz^{d-2}+dx^{d-1},\qquad
\frac{\partial f_d}{\partial y}=xz^{d-2}+dy^{d-1},\qquad
\frac{\partial f_d}{\partial z}=(d-2)xyz^{d-3}.
\]
A direct computation gives
\[
(d-2)y\frac{\partial f_d}{\partial y}-z\frac{\partial f_d}{\partial z}
=d(d-2)y^d,
\]
hence $y^d\in J_{f_d}$. Since $T-1=3d-7\ge d$ it follows that
$
y^{T-1}=y^{T-1-d}\cdot y^d\in (J_{f_d})_{T-1}.
$
Therefore
\[
[y^{T-1}]\in \PP((J_{f_d})_{T-1})\bigcap V_3^{T-1},
\]
so condition {\rm (II)} fails. Hence $X_d$ is not Veronese-Avoiding for $d\ge 4$.
\end{proof}
We now turn to the study of Veronese-Avoiding hypersurfaces in the two basic settings relevant to this paper. We first consider the smooth case, where the condition can be interpreted in terms of the Macaulay inverse system of the Milnor algebra. We then study the singular case, with emphasis on hypersurfaces having isolated singularities, and show that the Veronese-Avoiding property is closely related to the geometry of the singular locus.


\subsection{Smooth hypersurfaces}
Let $f \in R_d$ be a homogeneous polynomial of degree $d\geq 3$ defining a smooth hypersurface $X=V(f)\subseteq \PP_k^{n-1}.$
Then the gradient ideal $J_f$ is generated by a regular sequence of degree $d-1$. Hence the Milnor algebra $M_f:=R/J_f$
is a standard graded Artinian Gorenstein $k$-algebra, in fact a complete intersection, with socle degree $T:=n(d-2).$
Moreover, the socle $\Soc(M_f)=(M_f)_T$ is one-dimensional, and is generated by the class of the Hessian polynomial $\mathrm{Hess}(f)$; see, for instance, \cite[Lemma 3.3]{Saito}.

Since $M_f$ is a complete intersection of type $(d-1,\ldots,d-1)$, its Hilbert series is
\[
\mathrm{HS}_{M_f}(t)=\left(1+t+\cdots+t^{d-2}\right)^n.
\]
In particular, $\dim_k(M_f)_i=\dim_k(M_f)_{T-i}$ for all $0\leq i\leq T$,
and $\dim_k(M_f)_i=0 $ for all $i>T$.  
Since $\dim_k(M_f)_1=n$, duality yields
\begin{equation}\label{eq:codim-smooth}
\dim_k(M_f)_{T-1}=n.
\end{equation}
Equivalently,
\[
\codim_{\PP(R_{T-1})}(\PP((J_f)_{T-1})=n.
\]
Thus every smooth hypersurface satisfies condition {\rm (I)} of Definition~\ref{VAH}.

We now recall the language of Macaulay inverse systems. Let
\[
S=k[y_1,\ldots,y_n]
\]
be the ring of differential duals. The ring $R$ acts on $S$ by differentiation:
\[
(h\circ F)(y_1,\ldots,y_n)
:=
h\!\left(\frac{\partial}{\partial y_1},\ldots,\frac{\partial}{\partial y_n}\right)F(y_1,\ldots,y_n),
\qquad h\in R,\;F\in S.
\]
For each integer $e\geq 0$, this action induces a perfect pairing
\[
R_e\times S_e\longrightarrow k,
\qquad
(h,F)\longmapsto h\circ F,
\]
called the \emph{apolar pairing}.
Given a homogeneous polynomial $F\in S_e$, its apolar ideal is defined by
\[
I_F:=\Ann_R(F)=\{\,h\in R\mid h\circ F=0\,\}.
\]
Then $R/I_F$ is a standard graded Artinian Gorenstein $k$-algebra with socle degree $e$. Conversely, the following special case of Macaulay duality is classical; see, for example, \cite[Theorem 21.6]{EisenbudBook}.
\begin{Theorem}\label{thm:MIS}
For each $e\geq 0$, the correspondence
\[
[F]\longmapsto \Ann_R(F)
\]
gives a bijection between $\PP(S_e)$ and the set of homogeneous ideals $I\subseteq R$ such that $R/I$ is a standard graded Artinian Gorenstein algebra with socle degree $e$.
\end{Theorem}
Applied to the Milnor algebra $M_f$, Theorem~\ref{thm:MIS} shows that there exists a homogeneous polynomial
\[
F\in S_T,
\]
unique up to a nonzero scalar, such that
\[
J_f=\Ann_R(F).
\]
We call $F$ the \emph{Macaulay inverse system} of $M_f$.

The next theorem characterizes smooth Veronese-Avoiding hypersurfaces in terms of the Macaulay inverse system of their Milnor algebra. Although the result is due to Alper and Isaev \cite[Proposition 4.3]{Alper-Isaev1}, we include a new proof, based on Macaulay inverse systems, which fits naturally with the viewpoint developed in this paper.
\begin{Theorem}\label{nonsingularAssform}
Let $f \in R_d$ be a homogeneous polynomial of degree $d \geq 3$ defining a smooth hypersurface in $\PP_k^{n-1}$, and let $F\in S_T$ be the Macaulay inverse system of $M_f$, where $T=n(d-2)$. Then the following are equivalent:
\begin{enumerate}
\item[{\rm (a)}] The hypersurface $V(f)$ is Veronese-Avoiding.
\item[{\rm (b)}] The hypersurface $V(F)\subseteq \PP_k^{n-1}$ is smooth.
\end{enumerate}
\end{Theorem}

\begin{proof}
Since $V(f)$ is smooth, condition {\rm (I)} in Definition~\ref{VAH} holds by \eqref{eq:codim-smooth}. Thus $V(f)$ is Veronese-Avoiding if and only if condition {\rm (II)} holds.

Assume first that $V(f)$ is Veronese-Avoiding. Suppose, by contradiction, that $V(F)$ is singular. After a linear change of coordinates, we may assume that $p=[0:\cdots:0:1]\in \PP_k^{n-1}$ is a singular point of $V(F)$. Then the monomials
\[
y_n^T,\  y_1y_n^{T-1},\  \ldots,\  y_{n-1}y_n^{T-1}
\]
do not occur in $F$. Equivalently,
\[
x_n^T\circ F=0,
\qquad
x_i x_n^{T-1}\circ F=0
\quad\text{for }i=1,\ldots,n-1.
\]
Since $J_f=\Ann_R(F)$, it follows that
\[
x_n^T,\;x_i x_n^{T-1}\in J_f
\qquad (i=1,\ldots,n-1).
\]
Passing to the quotient $M_f=R/J_f$, we obtain
\[
\overline{x_n}^{\,T}=0,
\qquad
\overline{x_i}\,\overline{x_n}^{\,T-1}=0
\quad\text{for }i=1,\ldots,n-1.
\]
Hence $\overline{x_n}^{\,T-1}$ is annihilated by all linear forms, and therefore $\overline{x_n}^{\,T-1}\in \Soc(M_f)$. 
But $\Soc(M_f)$ is concentrated in degree $T$, whereas $\overline{x_n}^{\,T-1}$ has degree $T-1$. Thus $\overline{x_n}^{\,T-1}=0$. Therefore $x_n^{T-1}\in (J_f)_{T-1}$ so
\[
[x_n^{T-1}]\in \PP((J_f)_{T-1})\bigcap V_n^{T-1},
\]
contradicting condition {\rm (II)}. Hence $V(F)$ is smooth.

Conversely, assume that $V(F)$ is smooth. Suppose, by contradiction, that $V(f)$ is not Veronese-Avoiding. Since condition {\rm (I)} already holds, condition {\rm (II)} must fail. Hence there exists a nonzero linear form $\ell\in R_1$ such that $\ell^{T-1}\in (J_f)_{T-1}$. After a projective change of coordinates, we may assume that $\ell=x_n$. Then
\[
x_n^{T-1}\in J_f=\Ann_R(F).
\]
Since $\Ann_R(F)$ is an ideal, it also contains
\[
x_n^T,\  x_i x_n^{T-1}\quad (i=1,\ldots,n-1).
\]
Therefore
\[
x_n^T\circ F=0,
\qquad
x_i x_n^{T-1}\circ F=0
\quad\text{for }i=1,\ldots,n-1.
\]
As $F$ has degree $T$, this implies that the coefficients of the monomials
\[
y_n^T,\ y_1y_n^{T-1},\ \ldots,\ y_{n-1}y_n^{T-1}
\]
all vanish. Hence $[0:\cdots:0:1]$ is a singular point of $V(F)$, contradicting the smoothness of $V(F)$. Therefore $V(f)$ is Veronese-Avoiding.
\end{proof}

\begin{Remark}
Theorem~\ref{nonsingularAssform} gives a practical criterion for detecting the Veronese-Avoiding property in the smooth case. Once $f$ is smooth, condition {\rm (I)} is automatic, and condition {\rm (II)} is equivalent to the smoothness of the Macaulay inverse system of $M_f$. Thus one may test whether $V(f)$ is Veronese-Avoiding by studying the hypersurface defined by the inverse system, rather than by computing the intersection
$\PP((J_f)_{T-1})\bigcap V_n^{T-1}$. 
\end{Remark}
\begin{Example}[Smooth cubic curves in Hesse form]\label{ex:Hesse}
	Let
	\[
	f_\lambda=x_1^3+x_2^3+x_3^3-3\lambda x_1x_2x_3,
	\qquad \lambda^3\neq 1,
	\]
	so that $X_\lambda=V(f_\lambda)\subseteq \PP_k^2$ is a smooth cubic curve. Its gradient ideal is
	\[
	J_{f_\lambda}=
	(x_1^2-\lambda x_2x_3,\;
	x_2^2-\lambda x_1x_3,\;
	x_3^2-\lambda x_1x_2).
	\]
	Since the socle degree of $M_{f_\lambda}=R/J_{f_\lambda}$ is $3$, a Macaulay inverse system of $M_{f_\lambda}$ is, up to a nonzero scalar,
	\[
	F_\lambda=
	\lambda(y_1^3+y_2^3+y_3^3)+6y_1y_2y_3.
	\]
	Indeed,
	\[
	\left(x_1^2-\lambda x_2x_3\right)\circ F_\lambda
	=
	6\lambda y_1-\lambda\cdot 6y_1
	=
	0,
	\]
	and similarly for the other two generators.
	
	If $\lambda=0$, then $F_0=6y_1y_2y_3$, 	which is singular. If $\lambda\neq 0$, then, up to a nonzero scalar,
	\[
	F_\lambda
	=
	y_1^3+y_2^3+y_3^3-3\mu y_1y_2y_3,
	\qquad
	\mu=-\frac{2}{\lambda}.
	\]
	Thus $V(F_\lambda)$ is again a cubic in Hesse form, and hence it is smooth if and only if $	\mu^3\neq 1$, equivalently $\lambda^3\neq -8$.  
	Therefore $V(F_\lambda)$ is smooth if and only if $\lambda\neq 0$ and $	\lambda^3\neq -8$. 	Hence, by Theorem~\ref{nonsingularAssform}, the smooth cubic curve $X_\lambda$ is Veronese-Avoiding if and only if
	\[
	\lambda\neq 0
	\quad\text{and}\quad
	\lambda^3\neq -8.
	\]
\end{Example}
\begin{Example}[A symmetric quartic example]\label{ex:symmetric-quartic}
	Let $X=V(f)\subseteq \PP_k^2$, where 
	$$f=x_1^4+x_2^4+x_3^4+4x_1x_2x_3(x_1+x_2+x_3).
	$$
	Set
	\[
	e_1=x_1+x_2+x_3,\qquad e_2=x_1x_2+x_1x_3+x_2x_3,\qquad e_3=x_1x_2x_3.
	\]
We get
	\[
	x_1^4+x_2^4+x_3^4=e_1^4-4e_1^2e_2+2e_2^2+4e_1e_3,
	\]
	hence
	\[
	f=G(e_1,e_2,e_3),\qquad
	G(u,v,w)=u^4-4u^2v+2v^2+8uw.
	\]
	
	We first show that $X$ is smooth. If $p=[x_1:x_2:x_3]\in X$ is a singular point with pairwise distinct coordinates, then
	\[
	\begin{pmatrix}
		1 & x_2+x_3 & x_2x_3\\
		1 & x_1+x_3 & x_1x_3\\
		1 & x_1+x_2 & x_1x_2
	\end{pmatrix}
	\]
	is the Jacobian matrix of the change of variables $(x_1,x_2,x_3)\mapsto (e_1,e_2,e_3)$, and its determinant is
	\[
	(x_1-x_2)(x_1-x_3)(x_2-x_3)\neq 0.
	\]
	Hence, by the chain rule, the vanishing of the partial derivatives of $f$ at $p$ implies
	\[
	G_u(p)=G_v(p)=G_w(p)=0.
	\]
	Since
	\[
	G_u=4u^3-8uv+8w,\qquad G_v=-4u^2+4v,\qquad G_w=8u,
	\]
	these vanish simultaneously only at $u=v=w=0$, impossible in projective space. Therefore a singular point must have two equal coordinates. By symmetry, we may write it as $[a:a:b]$. Then
	\[
	\frac{\partial f}{\partial x_1}(a,a,b)=4a(a^2+3ab+b^2),\qquad
	\frac{\partial f}{\partial x_3}(a,a,b)=4(2a^3+2a^2b+b^3).
	\]
	If $a=0$, then $\partial f/\partial x_3=4b^3\neq 0$, so $a\neq 0$. After scaling $a=1$ and setting $t=b/a$, a singular point would have to satisfy
	\[
	t^2+3t+1=0,\qquad t^3+2t+2=0.
	\]
	Since
	\[
	\operatorname{Res}_t(t^2+3t+1,\;t^3+2t+2)=-25\neq 0,
	\]
	where $\operatorname{Res}_t(-,-)$ denotes the resultant with respect to $t$, these two equations have no common root. Thus $X$ is smooth.
	
	Now let
	\[
	s_1=y_1+y_2+y_3,\qquad
	s_2=y_1y_2+y_1y_3+y_2y_3,\qquad
	s_3=y_1y_2y_3.
	\]
	A direct apolar computation shows that a Macaulay inverse system of the Milnor algebra $M_f$ is
	\[
	F=
	6s_1^6-30s_1^4s_2-180s_1^3s_3+105s_1^2s_2^2+510s_1s_2s_3-190s_2^3-165s_3^2.
	\]
	
	We next show that $V(F)\subseteq \PP_k^2$ is smooth. Write $F=H(s_1,s_2,s_3)$, where
	\[
	H=
	6s_1^6-30s_1^4s_2-180s_1^3s_3+105s_1^2s_2^2+510s_1s_2s_3-190s_2^3-165s_3^2.
	\]
	If $q=[y_1:y_2:y_3]$ is a singular point with pairwise distinct coordinates, then the same Jacobian argument shows that
	\[
	H_{s_1}(q)=H_{s_2}(q)=H_{s_3}(q)=0.
	\]
	A Gr\"obner basis computation for the ideal $(H_{s_1},H_{s_2},H_{s_3})$ shows that it contains
	\[
	s_3^3,\ \  100586\,s_2^3-404547\,s_3^2,\ \ 6s_1^3-17s_1s_2+11s_3.
	\]
	Hence any common zero of $H_{s_1},H_{s_2},H_{s_3}$ satisfies $s_3=s_2=s_1=0$, impossible in projective space. Therefore a singular point of $V(F)$ must have two equal coordinates. By symmetry, write it as $[a:a:b]$. Then
	\[
	\frac{\partial F}{\partial y_1}(a,a,b)
	=
	6\bigl(67a^5-95a^4b+70a^3b^2-50a^2b^3+15ab^4+b^5\bigr),
	\]
	and
	\[
	\frac{\partial F}{\partial y_3}(a,a,b)
	=
	-12\bigl(19a^5-35a^4b+50a^3b^2-30a^2b^3-5ab^4-3b^5\bigr).
	\]
	If $a=0$, then $\partial F/\partial y_1=6b^5\neq 0$, so $a\neq 0$. After scaling $a=1$ and setting $t=b/a$, a singular point would have to satisfy
	\[
	P(t)=t^5+15t^4-50t^3+70t^2-95t+67=0,
	\]
	\[
	Q(t)=3t^5+5t^4+30t^3-50t^2+35t-19=0.
	\]
	Since
	\[
	\operatorname{Res}_t(P,Q)=-112990236800000\neq 0,
	\]
	these equations have no common root. Therefore $V(F)$ is smooth. By Theorem~\ref{nonsingularAssform}, the quartic curve $X$ is Veronese-Avoiding.
\end{Example}

\begin{Remark}
	The same method applies to the higher members of the family
	\[
	x_1^d+x_2^d+x_3^d+d\,x_1x_2x_3(x_1+x_2+x_3)^{d-3},
	\qquad d\ge 5.
	\]
	In each case, symmetry forces the Macaulay inverse system to be a weighted-homogeneous polynomial in the elementary symmetric functions $s_1,s_2,s_3$, and smoothness can be tested by the same two-step argument: first one excludes singular points with pairwise distinct coordinates by passing to the symmetric variables, and then one treats the case of two equal coordinates by a resultant computation. The formulas become substantially longer, so we record here only the quartic case.
\end{Remark}

\begin{Example}[Fermat hypersurfaces]\label{ex:fermatsmooth}
The Fermat hypersurface gives a simple illustration of Theorem~\ref{nonsingularAssform}. Indeed, if
$f=x_1^d+\cdots+x_n^d$, then $J_f=(x_1^{d-1},\ldots,x_n^{d-1})$, and the Macaulay inverse system of $M_f$ is, up to a nonzero scalar,
\[
F=(y_1\cdots y_n)^{d-2}.
\]
Since $V(F)$ is singular, Theorem~\ref{nonsingularAssform} shows again that the Fermat hypersurface is not Veronese-Avoiding.
\end{Example}

\begin{Corollary}\label{genericcase}
	For every integer $d\ge 3$, a generic polynomial $f\in R_d=k[x,y,z]_d$ defines a Veronese-Avoiding plane curve.
\end{Corollary}

\begin{proof}
	Fix $d\ge 3$. The Veronese-Avoiding condition is open in $R_d$, since it is given by the conjunction of the codimension condition and the avoidance condition, both of which are open in families.
	
	It therefore suffices to exhibit, for each degree $d\ge 3$, one polynomial in $R_d$ defining a Veronese-Avoiding plane curve. For $d=3$, this is provided by Example~\ref{ex:Hesse}. For $d=4$, this is given by Example~\ref{ex:symmetric-quartic}. For $d\ge 5$, the remark following Example~\ref{ex:symmetric-quartic} yields a family of smooth Veronese-Avoiding plane curves of degree $d$.
	
	Hence, for every $d\ge 3$, the set of degree-$d$ Veronese-Avoiding plane curves contains a nonempty open subset of $\PP(R_d)$. In other words, a generic polynomial in $R_d$ defines a Veronese-Avoiding plane curve.
\end{proof}

\begin{Remark}
	By~\cite[Lemma~4.5]{Alper-Isaev1}, the existence of a Veronese-Avoiding hypersurface of a fixed degree implies the existence of a smooth Veronese-Avoiding hypersurface of the same degree. Thus singular examples may be used as a bridge to produce smooth Veronese-Avoiding hypersurfaces. For this reason, the study of singular Veronese-Avoiding hypersurfaces is not merely auxiliary: it provides a natural and effective way to understand the smooth case as well. This motivates the next subsection, where we investigate the Veronese-Avoiding property for hypersurfaces with isolated singularities.
\end{Remark}

\subsection{Singular hypersurfaces}
Let $f\in R_d$ be a reduced homogeneous polynomial of degree $d\geq 3$ defining a hypersurface
\[
X=V(f)\subseteq \PP_k^{n-1}
\]
with only isolated singularities. The singular subscheme $\Sing(X)$ is defined by the gradient ideal $J_f$ which consists of finitely many points. Let $p\in \Sing(X)$ be a singular point. After a projective change of coordinates, we may assume that $p=[0:\cdots:0:1]$. On the affine chart $U_{x_n}\cong \AA_k^{n-1}$ with coordinates
$T_i=x_i/x_n$ for $i=1,\ldots,n-1$, the hypersurface $X$ is given by the polynomial
\[
F(T_1,\ldots,T_{n-1})=f(x_1,\ldots,x_{n-1},1)\in A:=k[T_1,\ldots,T_{n-1}].
\]
Let
\[
J_F=\left(\frac{\partial F}{\partial T_1},\ldots,\frac{\partial F}{\partial T_{n-1}}\right),
\qquad
I_F=(F,J_F),
\]
and let $\fp=(T_1,\ldots,T_{n-1})$ be the maximal ideal corresponding to the origin. The \emph{Milnor number} and \emph{Tjurina number} of $F$ at $\fp$ are defined by
\[
\mu_{\fp}(F):=\dim_k A_{\fp}/(J_F)_{\fp},
\qquad
\tau_{\fp}(F):=\dim_k A_{\fp}/(I_F)_{\fp}.
\]
The total Milnor and Tjurina numbers of $X$ are
\begin{equation}\label{eq:MTformula}
\mu(X)=\sum_{p\in \Sing(X)}\mu_p(X),
\qquad
\tau(X)=\sum_{p\in \Sing(X)}\tau_p(X).
\end{equation}
It is always true that
$
\tau(X)\leq \mu(X),
$
and equality holds if and only if all singularities are quasihomogeneous; see~\cite{Saito1}. For a recent algebraic characterization of isolated quasihomogeneous singularities, see~\cite{HAA}.

We now recall some notation for a finite set of points in projective space.
Let $\Gamma\subseteq \PP_k^{n-1}$ be a finite set of reduced points, and let
\[
I(\Gamma)=\bigcap_{p\in \Gamma} I(p)
\]
be its defining ideal. For each $m\geq 0$, consider the linear system
\[
R_m(\Gamma):=\{\,h\in R_m\mid h(p)=0\text{ for all }p\in \Gamma\,\}.
\]
Its \emph{defect} is defined by
\[
\mathrm{def} R_m(\Gamma):=|\Gamma|-\codim(R_m(\Gamma)).
\]
We shall use only the case $m=1$, where
\begin{equation}\label{eq:defPoints}
\mathrm{def} R_1(\Gamma)=|\Gamma|-n+\dim_k R_1(\Gamma).
\end{equation}
Thus
$
\mathrm{def} R_1(\Gamma)=0
$
if and only if the points of $\Gamma$ impose independent linear conditions on $R_1$. Equivalently, $|\Gamma|\leq n$ and $\Gamma$ spans a $\PP_k^{|\Gamma|-1}$. 
In particular, when $|\Gamma|=n$, this is equivalent to saying that the points of $\Gamma$ are in general linear position.

Let $J\subseteq R=k[x_1,\ldots,x_n]$ be a homogeneous ideal, and let $\fm=(x_1,\ldots,x_n)$ be the irrelevant maximal ideal. The \emph{saturation} of $J$ is defined by
\[
J^{\mathrm{sat}}:=(J:_R\fm^\infty)=\{\,h\in R \mid \fm^t h\subseteq J \text{ for some } t \ge 0\};
\]
see \cite[15.10.6]{EisenbudBook}. The ideal $J^{\mathrm{sat}}$ is homogeneous, and $J$ and $J^{\mathrm{sat}}$ define the same closed subscheme of $\PP_k^{n-1}$. We say that $J$ is \emph{saturated} if $J=J^{\mathrm{sat}}$. In particular, if $V(J)$ is reduced, then $J^{\mathrm{sat}}=I(V(J))=\sqrt{J}$; see \cite[Section~2]{A.Dimca1}.

For the singular subscheme $\Sing(X)=V(J_f)$, Dimca defines the \emph{defect} in degree $m$ by
\[
\mathrm{def}_m(\Sing(X)):=\tau(X)-\dim_k\!\left(\frac{R_m}{(J_f^{\mathrm{sat}})_m}\right).
\]
In degree $m=1$, this becomes
\begin{equation}\label{def1Sing(X)}
\mathrm{def}_1(\Sing(X))
=\tau(X)-\dim_k\!\left(\frac{R_1}{(J_f^{\mathrm{sat}})_1}\right)
=\tau(X)-n+\dim_k (J_f^{\mathrm{sat}})_1.
\end{equation}

We also recall the \emph{coincidence threshold}
\[
\mathrm{ct}(X):=\max\left\{ q \ \middle|\ \dim_k(M_f)_m=\dim_k(M_{f_s})_m \text{ for all } m\le q \right\},
\]
where $f_s$ is a homogeneous polynomial of degree $d$ defining a smooth hypersurface in $\PP_k^{n-1}$. By \cite[Proposition~2]{A.Dimca1}, one has
\begin{equation}\label{defY}
	\mathrm{def}_1(\Sing(X))=0
	\quad\Longleftrightarrow\quad
	\mathrm{ct}(X)\geq T=n(d-2).
\end{equation}

We have already seen that every smooth hypersurface satisfies condition {\rm (I)} of Definition~\ref{VAH}. For hypersurfaces with isolated singularities, the problem of characterizing the same condition leads to the following basic lemma.

\begin{Lemma}\label{lem:HFMilnor}
	Let $X=V(f)\subseteq \PP_k^{n-1}$ be a hypersurface of degree $d\geq 3$ with only isolated singularities, and set $T=n(d-2)$. Then the following are equivalent:
	\begin{enumerate}
		\item[{\rm (a)}] $\codim \PP((J_f)_{T-1})=n$;
		\item[{\rm (b)}] $\mathrm{def}_1(\Sing(X))=0$;
		\item[{\rm (c)}] $\mathrm{ct}(X)\geq T$.
	\end{enumerate}
	If, moreover, $X$ is nodal with singular set $\Gamma$, then these conditions are also equivalent to $\mathrm{def} R_1(\Gamma)=0$. In particular, if $X$ has exactly $n$ nodes, then conditions {\rm (a)}--{\rm (c)} are equivalent to the singular points being in general linear position.
\end{Lemma}
\begin{proof}
Since
\[
\codim_{\PP(R_{T-1})}\PP((J_f)_{T-1})=\dim_k(M_f)_{T-1},
\]
the equivalence between {\rm (a)} and {\rm (b)} follows from \cite[Theorem~1]{A.Dimca1}. The equivalence between {\rm (b)} and {\rm (c)} is \cite[Proposition~2]{A.Dimca1}. If $X$ is nodal, then $\tau(X)=|\Gamma|$ and $J_f^{\mathrm{sat}}=I(\Gamma)$, so \eqref{def1Sing(X)} gives
\[
\mathrm{def}_1(\Sing(X))
=
|\Gamma|-n+\dim_k I(\Gamma)_1
=
\mathrm{def}\,R_1(\Gamma).
\]
This proves the remaining assertions.
\end{proof}
For the main theorem, we shall also use the Jacobian module
\[
N(f):=H^0_{\fm}(M_f)=J_f^{\mathrm{sat}}/J_f.
\]
When $X$ has only isolated singularities, this module satisfies the graded self-duality
\begin{equation}\label{eq:self-duality}
N(f)_q^\vee \simeq N(f)_{T-q}
\qquad\text{for all }q\in \mathbb Z.
\end{equation}
We now describe the Veronese-Avoiding property in the singular setting. We begin with singular cubic curves, where one obtains a complete characterization, and then turn to hypersurfaces with exactly $n$ isolated singular points.

\begin{Proposition}\label{prop:cubic-plane-classification}
	Let $R=k[x,y,z]$, and let $X=V(f)\subseteq \PP_k^2$ be a reduced singular cubic curve. Then the following are equivalent:
	\begin{enumerate}
		\item[{\rm (a)}] $X$ is Veronese-Avoiding;
		\item[{\rm (b)}] $X$ is nodal.
	\end{enumerate}
\end{Proposition}

\begin{proof}
	Since $d=3$ and $n=3$, one has $T=3$ and hence $T-1=2$. Thus $X$ is Veronese-Avoiding if and only if $\codim \PP((J_f)_2)=3$ and $\PP((J_f)_2)\bigcap V_3^2=\emptyset$.
	
	\smallskip
	
	\noindent
	{ (a)$\Rightarrow$(b).}
	Assume that $X$ is Veronese-Avoiding. Suppose that $X$ is not nodal. Then $X$ has a singular point $p$ which is not an ordinary double point. After a projective change of coordinates, we may assume that $p=[0:0:1]$. Since $X$ is cubic and singular at $p$, we may write $f=z\,q(x,y)+c(x,y)$, where $q\in k[x,y]_2$ and $c\in k[x,y]_3$. The point $p$ is nodal if and only if $q$ is reduced, that is, the product of two distinct linear forms. Since $p$ is not nodal, $q$ must be a square, say $q=\ell^2$ for some nonzero linear form $\ell\in k[x,y]_1$. But then $\partial f/\partial z=q=\ell^2$, so $[\ell^2]\in \PP((J_f)_2)\bigcap V_3^2$, contradicting condition {\rm (II)} in Definition~\ref{VAH}. Therefore $X$ is nodal.
	
	\smallskip
	
	\noindent
	{ (b)$\Rightarrow$(a).}
	Assume that $X$ is a reduced nodal cubic. We distinguish the possible cases.
	
	\medskip
	
	\noindent
	{\bf Case 1:} $X$ is irreducible.
	Then $X$ has exactly one node and is projectively equivalent to $f=xyz+x^3+y^3$. By Proposition~\ref{prop:family-VAH-iff-cubic}, this curve is Veronese-Avoiding.
	
	\medskip
	
	\noindent
	{\bf Case 2:} $X$ is the union of a line and a smooth conic meeting transversely.
	Then $X$ is projectively equivalent to $f=z(xy-z^2)=xyz-z^3$. Its gradient ideal is $J_f=(yz,\;xz,\;xy-3z^2)$, hence $(J_f)_2=\langle yz,\;xz,\;xy-3z^2\rangle_k$. Thus $\dim_k(J_f)_2=3$, so $\codim \PP((J_f)_2)=3$. If a square $(ax+by+cz)^2$ lies in $(J_f)_2$, then comparing coefficients with the basis $x^2,xy,xz,y^2,yz,z^2$ shows that $a=b=0$. Hence the square reduces to $c^2z^2$, but no nonzero multiple of $z^2$ lies in $\langle yz,\;xz,\;xy-3z^2\rangle_k$ without also introducing an $xy$-term. Therefore $c=0$, and so $\PP((J_f)_2)\bigcap V_3^2=\emptyset$. Thus $X$ is Veronese-Avoiding.
	
	\medskip
	
	\noindent
	{\bf Case 3:} $X$ is the union of three distinct lines in general position.
	Then $X$ is projectively equivalent to $f=xyz$. Its gradient ideal is $J_f=(yz,\;xz,\;xy)$, so $(J_f)_2=\langle yz,\;xz,\;xy\rangle_k$. Again $\codim \PP((J_f)_2)=3$. If a square $(ax+by+cz)^2$ lies in $(J_f)_2$, then its $x^2$, $y^2$, and $z^2$ coefficients must all vanish, hence $a=b=c=0$. Thus $\PP((J_f)_2)\bigcap V_3^2=\emptyset$. Therefore $X$ is Veronese-Avoiding.
	
	Combining the three cases, every reduced nodal cubic is Veronese-Avoiding.
\end{proof}

\begin{Theorem}\label{thm:main-singular}
	Let $X=V(f)\subseteq \PP_k^{n-1}$ be a reduced hypersurface of degree $d\geq 3$ with only isolated singularities. Assume that $\Sing(X)$ consists of exactly $n$ points. Then the following are equivalent:
	\begin{enumerate}
		\item[{\rm (a)}] $X$ is Veronese-Avoiding;
		\item[{\rm (b)}] $\Sing(X)$ consists of $n$ nodes in general linear position.
	\end{enumerate}
\end{Theorem}

\begin{proof}
\noindent
{\rm (b)$\Rightarrow$(a).}
Let $\Gamma=\Sing(X)=\{p_1,\ldots,p_n\}$. By Lemma~\ref{lem:HFMilnor}, condition {\rm (I)} in Definition~\ref{VAH} holds. It remains to prove condition {\rm (II)}.

After a projective change of coordinates, we may assume that the points of $\Gamma$ are the coordinate points in $\PP_k^{n-1}$. Let $I=I(\Gamma)$. Since all singularities are nodes, the singular scheme is reduced, and hence $J_f^{\mathrm{sat}}=I$. Therefore
\[
N(f)=J_f^{\mathrm{sat}}/J_f=I/J_f.
\]
No nonzero linear form vanishes at all coordinate points, so $I_1=0$. Since $J_f$ is generated in degree $d-1\geq 2$, we also have $(J_f)_1=0$. Thus $N(f)_1=0$. By the self-duality \eqref{eq:self-duality}, $N(f)_{T-1}=0$, and therefore
\[
(J_f)_{T-1}=I_{T-1}.
\]

A form of degree $T-1$ vanishes at all coordinate points if and only if the coefficients of the pure powers
\[
x_1^{T-1},\ldots,x_n^{T-1}
\]
are all zero. Hence $I_{T-1}$ is spanned by all degree-$(T-1)$ monomials except the pure powers. In particular,
\[
\codim_{\PP(R_{T-1})}\PP(I_{T-1})=n.
\]
Now suppose that $[\ell^{T-1}]\in \PP(I_{T-1})\cap V_n^{T-1}$, where
\[
\ell=a_1x_1+\cdots+a_nx_n.
\]
Then the coefficients of $x_1^{T-1},\ldots,x_n^{T-1}$ in $\ell^{T-1}$ vanish, so
\[
a_1^{T-1}=\cdots=a_n^{T-1}=0.
\]
Since $\operatorname{char}(k)=0$, this gives $a_1=\cdots=a_n=0$, a contradiction. Therefore
\[
\PP(I_{T-1})\bigcap V_n^{T-1}=\emptyset.
\]
Since $(J_f)_{T-1}=I_{T-1}$, condition {\rm (II)} holds. Hence $X$ is Veronese-Avoiding.

	\medskip
	
	\noindent
	{\rm  (a)$\Rightarrow$(b).}
	Assume that $X$ is Veronese-Avoiding, and write $\Gamma=\Sing(X)=\{p_1,\dots,p_n\}$. Let $I:=J_f^{\mathrm{sat}}$. Since $X$ is Veronese-Avoiding, condition {\rm (I)} holds. Hence, by Lemma~\ref{lem:HFMilnor}, $\mathrm{def}_1(\Sing(X))=0$. By definition,
	\[
	\mathrm{def}_1(\Sing(X))=\tau(X)-\dim_k\!\left(\frac{R_1}{I_1}\right)=\tau(X)-n+\dim_k(I_1),
	\]
	so
	\begin{equation}\label{eq:tau-bound}
		\tau(X)=n-\dim_k(I_1)\leq n.
	\end{equation}
	On the other hand, $\Gamma$ consists of exactly $n$ singular points, and every isolated singular point has local Tjurina number at least one. Hence $\tau(X)=\sum_{p\in \Gamma}\tau_p(X)\geq n$. Combining this with \eqref{eq:tau-bound}, we obtain $\tau(X)=n$ and $\dim_k(I_1)=0$.
	
	Since $\tau(X)=n$ and $|\Gamma|=n$, it follows that $\tau_p(X)=1$ for every $p\in \Gamma$. Fix $p\in \Gamma$, and choose affine coordinates so that $p$ corresponds to the origin and $X$ is locally defined by $F\in k[[T_1,\dots,T_{n-1}]]$ with no linear term. Since $\tau_p(X)=1$, the local Tjurina algebra
	\[
	k[[T_1,\dots,T_{n-1}]]/(F,\partial F/\partial T_1,\dots,\partial F/\partial T_{n-1})
	\]
	has $k$-dimension one. Therefore the Tjurina ideal is the maximal ideal $(T_1,\dots,T_{n-1})$. Passing modulo $\fm^2$, one sees that the linear parts of $\partial F/\partial T_1,\dots,\partial F/\partial T_{n-1}$ span $\fm/\fm^2$. Equivalently, the quadratic part of $F$ is nondegenerate. Hence $p$ is an ordinary double point, that is, a node.
	
	Finally, since $I_1=0$, there is no nonzero linear form vanishing on all points of $\Gamma$. For a set of exactly $n$ points in $\PP_k^{n-1}$, this is equivalent to saying that the points are in general linear position. Thus $\Gamma$ consists of $n$ nodes in general linear position.
\end{proof}

\begin{Remark}
The preceding theorem gives a geometric interpretation of the construction used
by Alper--Isaev in the proof of \cite[Proposition~4.3]{Alper-Isaev1}. In their
notation the degree is denoted by $m$; here we write it as $d$. They consider
the auxiliary form
\[
f_0=
\begin{cases}
\displaystyle \sum_{1\leq i<j<k\leq n}x_ix_jx_k, & d=3,\\[6pt]
\displaystyle \sum_{1\leq i<j\leq n}
\left(x_i^{d-2}x_j^2+x_i^2x_j^{d-2}\right), & d\geq 4.
\end{cases}
\]
Let $T=n(d-2)$. In the proof of \cite[Proposition~4.3]{Alper-Isaev1},
they show that $(J_{f_0})_{T-1}$ has codimension $n$ in $R_{T-1}$ and
that
\[
(J_{f_0})_{T-1}\cap \{L^{T-1}:L\in R_1\}=0.
\]
By \cite[Lemma~4.4]{Alper-Isaev1}, this is equivalent to the
non-vanishing of the discriminant of the associated form of $f_0$.
In our terminology, this says exactly that $f_0$ is Veronese-Avoiding.

On the other hand, the coordinate points $p_1,\ldots,p_n$ are ordinary
nodes of $V(f_0)$ and they are in general linear position. Thus our
main theorem explains the Alper--Isaev construction conceptually: their
auxiliary form lies on the nodal boundary of the smooth Veronese-Avoiding
locus, and this boundary point is precisely of the type classified above.
\end{Remark}

\begin{Remark}
	Proposition~\ref{prop:cubic-plane-classification} gives a complete description of the singular cubic curves in $\PP_k^2$ that are Veronese-Avoiding: they are exactly the nodal ones. Theorem~\ref{thm:main-singular}, on the other hand, applies in every dimension and every degree $d\geq 3$, under the additional assumption that the singular locus consists of exactly $n$ isolated points. In that setting, the Veronese-Avoiding property is completely determined by the geometry of the singular locus: it is equivalent to requiring that the singularities be $n$ nodes in general linear position. 
\end{Remark}
\begin{Example}\label{ex:one-node-VAH-family}
	Let $d\ge 4$, and consider
	\[
	f=xyz^{d-2}+x^d+y^d+x^{d-1}z \in R=k[x,y,z].
	\]
	Then the plane curve $X=V(f)\subseteq \PP_k^2$ has a unique singular point, namely $[0:0:1]$, and this singularity is a node. We claim that $X$ is Veronese-Avoiding.
	
	Since a single reduced point imposes one independent linear condition on $R_1$, one has $\mathrm{def}\,R_1(\Sing(X))=0$. Therefore, by Lemma~\ref{lem:HFMilnor}, condition {\rm (I)} in Definition~\ref{VAH} holds.
	
	It remains to prove condition {\rm (II)}. Set $m=T-1=3d-7$, and suppose that $\ell^m\in (J_f)_m$ for some nonzero linear form $\ell=ax+by+cz$. Since every monomial occurring in the generators $f_x$, $f_y$, and $f_z$ is divisible by $x$ or by $y$, the same is true for every element of $(J_f)_m$. In particular, $(J_f)_m$ contains no nonzero multiple of $z^m$, so the coefficient of $z^m$ in $\ell^m$ must vanish. Hence $c=0$, and therefore $\ell=ax+by$.
	
	Now write
	\[
	(ax+by)^m=A f_x+B f_y+C f_z,
	\qquad \deg A=\deg B=\deg C=2d-6.
	\]
	Comparing coefficients and descending on the $z$-degree, one finds successively that the coefficients of $A$, $B$, and $C$ are forced to vanish step by step unless $a=b=0$. Hence there is no nonzero linear form $\ell$ such that $\ell^m\in (J_f)_m$. Equivalently, $\PP((J_f)_m)\bigcap V_3^m=\emptyset$.
	
	Thus condition {\rm (II)} also holds, and the claim follows.
\end{Example}
	\begin{Remark}\label{rem:global-phenomenon}
		The hypothesis $|\Sing(X)|=n$ in Theorem~\ref{thm:main-singular} is essential. More generally, when the singular locus has fewer than $n$ points, the Veronese-Avoiding property is not determined solely by the number of singular points or by their local analytic type.
		
		Indeed, if $|\Sing(X)|<n$, then there exists a nonzero linear form vanishing on $\Sing(X)$, so $I(\Sing(X))_1\neq 0$. Since $J_f$ is generated in degree $d-1\ge 2$, one has $(J_f)_1=0$, and therefore $N(f)_1\neq 0$, where $N(f)=J_f^{\mathrm{sat}}/J_f$ is the Jacobian module. By the self-duality of $N(f)$, it follows that $N(f)_{T-1}\neq 0$. Equivalently, $(J_f)_{T-1}$ is, in general, a proper subspace of $(J_f^{\mathrm{sat}})_{\,T-1}=I(\Sing(X))_{T-1}$. Thus the degree-$(T-1)$ piece of the Jacobian ideal is no longer determined only by the singular locus, and the Veronese-Avoiding condition depends on finer global features of the defining equation.
		
		This phenomenon already appears in the family
        \[f_d=xyz^{d-2}+x^d+y^d \qquad\text{and}\qquad
        g_d=xyz^{d-2}+x^d+y^d+x^{d-1}z,
        \] where $d\geq 4$.
		Both $V(f_d)$ and $V(g_d)$ have a unique singular point at $[0:0:1]$, and this singularity is a node. However, by Proposition~\ref{prop:family-VAH-iff-cubic}, the curve $V(f_d)$ is not Veronese-Avoiding for $d\ge 4$, whereas Example~\ref{ex:one-node-VAH-family} shows that $V(g_d)$ is Veronese-Avoiding. Hence two hypersurfaces may have the same number of singular points and the same local analytic type, while only one of them is Veronese-Avoiding.
		
		In other words, outside the range covered by Theorem~\ref{thm:main-singular}, the Veronese-Avoiding condition is a genuinely global property of the Jacobian ideal, rather than a condition depending only on the singularity type.
	\end{Remark}
\begin{Example}\label{ex:one-node-map}
	Consider the quartic
	\[
	f=xyz^2+x^4+y^4+x^3z \in R=k[x,y,z].
	\]
	Then $X=V(f)\subseteq \PP_k^2$ has a unique singular point at $p=[0:0:1]$, and this singularity is a node. Let $I=I(p)=(x,y)$. Since $p$ is a single reduced point, one has $\mathrm{def}\,R_1(\Sing(X))=0$, so condition {\rm (I)} holds by Lemma~\ref{lem:HFMilnor}.
	
	Now $m=T-1=5$, and $I_1=\langle x,y\rangle_k$. Since $(J_f)_1=0$, one has $N(f)_1=I_1$, and by self-duality $\dim_k N(f)_5=2$. Thus the class map
	\[
	\phi_f:\PP(I_1)\dashrightarrow \PP(N(f)_5),\qquad [\ell]\longmapsto [\,\ell^5 \bmod (J_f)_5\,],
	\]
	may be viewed, after choosing bases, as a rational map from $\PP^1$ to $\PP^1$. As shown above, there is no nonzero linear form $\ell=ax+by$ such that $\ell^5\in (J_f)_5$. Therefore $\phi_f$ is everywhere defined, condition {\rm (II)} holds, and $X$ is Veronese-Avoiding.
	
	This should be compared with the quartic
	\[
	f_0=xyz^2+x^4+y^4,
	\]
	for which $x^4,y^4\in J_{f_0}$, hence $x^5,y^5\in (J_{f_0})_5$. In this case the corresponding map $\phi_{f_0}$ has base points, and $V(f_0)$ is not Veronese-Avoiding.
\end{Example}

The preceding example suggests a more general interpretation of the Veronese-Avoiding condition when the singular locus consists of fewer than $n$ nodes. In this situation, condition {\rm (I)} is automatic as soon as the nodes impose independent linear conditions, and the essential issue becomes whether the map defined by $\ell\mapsto \ell^{\,T-1}$ acquires base points modulo the Jacobian ideal.

\begin{Theorem}\label{thm:nodal-r-points}
	Let $X=V(f)\subseteq \PP_k^{n-1}$ be a reduced hypersurface of degree $d\ge 3$ with isolated singularities, and let $\Gamma=\Sing(X)=\{p_1,\ldots,p_r\}$ with $r<n$. Assume that $\Gamma$ consists of $r$ nodes imposing independent linear conditions on $R_1$. Set $m=T-1=n(d-2)-1$ and $I:=I(\Gamma)$. Then the following are equivalent:
	\begin{enumerate}
		\item[{\rm (a)}] $X$ is Veronese-Avoiding;
		\item[{\rm (b)}] there is no nonzero linear form $\ell\in I_1$ such that $\ell^m\in (J_f)_m$;
		\item[{\rm (c)}] the rational map
		\[
		\phi_f:\PP(I_1)\dashrightarrow \PP(N(f)_m),\qquad [\ell]\longmapsto [\,\ell^m \bmod (J_f)_m\,],
		\]
		is everywhere defined.
	\end{enumerate}
	Moreover, $\dim_k I_1=\dim_k N(f)_m=n-r$, so after choosing bases, $\phi_f$ may be viewed as a rational map
	\[
	\PP^{\,n-r-1}\dashrightarrow \PP^{\,n-r-1}.
	\]
\end{Theorem}

\begin{proof}
	Since $\Gamma$ consists of $r$ nodes, one has $\tau(X)=r$ and $J_f^{\mathrm{sat}}=I(\Gamma)=I$. Since the points impose independent linear conditions on $R_1$, one has $\dim_k I_1=n-r$. Therefore
	\[
	\mathrm{def}_1(\Sing(X))=\tau(X)-n+\dim_k I_1=r-n+(n-r)=0.
	\]
	By Lemma~\ref{lem:HFMilnor}, condition {\rm (I)} in Definition~\ref{VAH} holds. Hence $X$ is Veronese-Avoiding if and only if condition {\rm (II)} holds.
	
	Now suppose that $\ell^m\in (J_f)_m$ for some nonzero linear form $\ell\in R_1$. Since $(J_f)_m\subseteq I_m$ and $I=I(\Gamma)$ is radical, it follows that $\ell$ vanishes at every point of $\Gamma$, that is, $\ell\in I_1$. Thus condition {\rm (II)} is equivalent to saying that there is no nonzero $\ell\in I_1$ with $\ell^m\in (J_f)_m$. This proves the equivalence of {\rm (a)} and {\rm (b)}.
	
	Since $N(f)_m=I_m/(J_f)_m$, the class of $\ell^m$ defines a point of $\PP(N(f)_m)$ unless $\ell^m\in (J_f)_m$. Hence the base locus of $\phi_f$ is exactly the set of points $[\ell]\in \PP(I_1)$ such that $\ell^m\in (J_f)_m$. This proves the equivalence of {\rm (b)} and {\rm (c)}.
	
	Finally, since $(J_f)_1=0$, one has $N(f)_1=I_1$, and by the self-duality \eqref{eq:self-duality},
	\[
	\dim_k N(f)_m=\dim_k N(f)_1=\dim_k I_1=n-r.
	\]
\end{proof}

\begin{Remark}\label{rem:nodal-r-points}
	Theorem~\ref{thm:nodal-r-points} should be viewed as the natural extension of Theorem~\ref{thm:main-singular} to the case $r<n$. In this range, once the singular locus consists of $r$ nodes imposing independent linear conditions on $R_1$, condition {\rm (I)} is automatic, and the Veronese-Avoiding property is controlled entirely by the base locus of the map $\phi_f$. In particular, for $r=1$ one obtains a rational map from $\PP^{n-2}$ to $\PP^{n-2}$, while for $r=n-1$ the criterion reduces to the single condition that, if $H$ is an equation of the unique hyperplane through the nodes, then $H^{\,m}\notin (J_f)_m$.
\end{Remark}
\section{The parameter space of Veronese-Avoiding hypersurfaces}\label{Pram-Space-VAH}

In this section we study the parameter space of Veronese-Avoiding hypersurfaces of fixed degree.
Fix integers $n\ge 3$ and $d\ge 3$, and let $R=k[x_1,\ldots,x_n]$. Set
\[
T:=n(d-2),\qquad N_d:=\dim \PP(R_d)=\binom{n+d-1}{d}-1.
\]
We denote by
\[
\mathscr V_{n,d}:=
\{\, [f]\in \PP(R_d)\mid V(f)\subseteq \PP_k^{n-1}\text{ is Veronese-Avoiding}\,\}
\]
the parameter space of Veronese-Avoiding hypersurfaces of degree $d$. We also set
\[
\mathscr V_{n,d}^{\mathrm{red}}:=\mathscr V_{n,d}\bigcap \mathscr R_{n,d},
\qquad
\mathscr V_{n,d}^{\mathrm{sm}}:=\mathscr V_{n,d}\bigcap \mathscr U_{n,d},
\]
where $\mathscr R_{n,d}\subseteq \PP(R_d)$ is the open locus of reduced hypersurfaces and $\mathscr U_{n,d}\subseteq \PP(R_d)$ is the open locus of smooth hypersurfaces.

We first describe $\mathscr V_{n,d}$ in terms of a Grassmannian. Set $m:=\dim_k R_{T-1}$ and $s:=\dim_k R_{T-d}$. For each $f\in R_d$, multiplication by the partial derivatives of $f$ induces a linear map
\[
\mu_f:R_{T-d}^{\oplus n}\to R_{T-1},\qquad
(a_1,\ldots,a_n)\mapsto \sum_{i=1}^n a_i\frac{\partial f}{\partial x_i}.
\]
Its image is precisely $(J_f)_{T-1}$. After choosing bases, $\mu_f$ is represented by a matrix
\[
A(f)\in \mathrm{Mat}_{m\times ns}(k),
\]
whose entries depend linearly on the coefficients of $f$.

\begin{Proposition}\label{Prop:rank-stratum-VAH}
	The locus
	\[
	\mathscr C_{n,d}:=
	\{\, [f]\in \PP(R_d)\mid \codim_{\PP(R_{T-1})}\PP((J_f)_{T-1})=n\,\}
	\]
	is a locally closed subvariety of $\PP(R_d)$.
\end{Proposition}

\begin{proof}
	Condition {\rm (I)} in Definition~\ref{VAH} is equivalent to $\dim_k (J_f)_{T-1}=m-n$, that is, to $\rank A(f)=m-n$. Since the entries of $A(f)$ depend linearly on the coefficients of $f$, this is a determinantal rank condition. Hence $\mathscr C_{n,d}$ is locally closed.
\end{proof}

Let $G:=\mathrm{Gr}(m-n,R_{T-1})$ be the Grassmannian of $(m-n)$-dimensional vector subspaces of $R_{T-1}$. On the stratum $\mathscr C_{n,d}$, the image of $\mu_f$ defines a morphism
\[
\gamma_{n,d}:\mathscr C_{n,d}\to G,\qquad [f]\mapsto (J_f)_{T-1}.
\]
Now let
\[
\Sigma_{n,d}:=
\{\,L\in G\mid \PP(L)\bigcap V_n^{T-1}\neq \emptyset\,\}.
\]

\begin{Proposition}\label{Prop:sigma-closed-VAH}
	The subset $\Sigma_{n,d}$ is closed in $G$. In particular,
	\[
	\Omega_{n,d}:=G\setminus \Sigma_{n,d}
	\]
	is a Zariski-open subset of $G$.
\end{Proposition}

\begin{proof}
	Consider the incidence correspondence
	\[
	\mathcal I:=
	\{\, (L,p)\in G\times V_n^{T-1}\mid p\in \PP(L)\,\}.
	\]
	This is closed in $G\times V_n^{T-1}$. Since $V_n^{T-1}$ is projective, the projection $\pi_1:\mathcal I\to G$ is proper. Therefore its image $\pi_1(\mathcal I)=\Sigma_{n,d}$ is closed.
\end{proof}

\begin{Corollary}\label{cor:locally-closed-VAH}
	The parameter space $\mathscr V_{n,d}$ is a locally closed subvariety of $\PP(R_d)$. More precisely,
	\[
	\mathscr V_{n,d}=\gamma_{n,d}^{-1}(\Omega_{n,d})\subseteq \mathscr C_{n,d}.
	\]
	In particular, condition {\rm (II)} is open on the stratum where condition {\rm (I)} holds.
\end{Corollary}

\begin{proof}
	This follows immediately from Propositions~\ref{Prop:rank-stratum-VAH} and~\ref{Prop:sigma-closed-VAH}.
\end{proof}

\begin{Remark}
	The variety $\mathscr V_{n,d}$ is stable under the natural action of $PGL_n$ on $\PP(R_d)$. Indeed, both the rank condition defining $\mathscr C_{n,d}$ and the avoidance condition defining $\Omega_{n,d}$ are invariant under projective changes of coordinates.
\end{Remark}

We now study the part of $\mathscr V_{n,d}$ coming from hypersurfaces with exactly $n$ nodes. Let $U_n\subseteq (\PP_k^{n-1})^n$ be the open subset of ordered $n$-tuples $(p_1,\ldots,p_n)$ in general linear position. For $\underline p=(p_1,\ldots,p_n)\in U_n$, set $Z_{\underline p}:=\{p_1,\ldots,p_n\}\subseteq \PP_k^{n-1}$. Consider the incidence variety
\[
\mathcal L_{n,d}:=
\{\,([f],\underline p)\in \PP(R_d)\times U_n \mid f\in H^0(\PP_k^{n-1},\mathcal I_{Z_{\underline p}}^2(d))\,\}.
\]
Thus $\mathcal L_{n,d}$ parametrizes degree-$d$ hypersurfaces singular at the ordered set of points $\underline p$.

\begin{Lemma}\label{lem:linear-system-dim}
	For every $\underline p\in U_n$, one has
	\[
	\dim \PP\bigl(H^0(\PP_k^{n-1},\mathcal I_{Z_{\underline p}}^2(d))\bigr)=N_d-n^2.
	\]
	Consequently, $\mathcal L_{n,d}$ is an irreducible projective bundle over $U_n$ of dimension
	\[
	\dim \mathcal L_{n,d}=N_d-n.
	\]
\end{Lemma}

\begin{proof}
	By projective equivalence, it is enough to consider the case where $Z_{\underline p}$ is the set of coordinate points. Then a form of degree $d$ is singular at all these points if and only if the coefficients of the monomials $x_i^d$ and $x_i^{d-1}x_j$ for $i\neq j$ all vanish. These are exactly $n+n(n-1)=n^2$ independent linear conditions. Hence
	\[
	\dim \PP\bigl(H^0(\PP_k^{n-1},\mathcal I_{Z_{\underline p}}^2(d))\bigr)=N_d-n^2.
	\]
	Since $U_n$ is irreducible of dimension $n(n-1)$, the conclusion follows.
\end{proof}

Let $\mathcal N_{n,d}\subseteq \mathcal L_{n,d}$ be the subset of pairs $([f],\underline p)$ such that $V(f)$ has an ordinary double point at each point of $\underline p$ and is smooth away from $Z_{\underline p}$.

\begin{Proposition}\label{prop:nodal-incidence}
	The subset $\mathcal N_{n,d}$ is a nonempty open subset of $\mathcal L_{n,d}$. In particular, $\mathcal N_{n,d}$ is irreducible of dimension
	\[
	\dim \mathcal N_{n,d}=N_d-n.
	\]
\end{Proposition}

\begin{proof}
	Fix $\underline p\in U_n$, and by projective equivalence assume that $Z_{\underline p}$ is the set of coordinate points. The fiber of $\mathcal L_{n,d}\to U_n$ over $\underline p$ is the linear system $\PP(H^0(\PP_k^{n-1},\mathcal I_{Z_{\underline p}}^2(d)))$.
	
	The base locus of this system is exactly $Z_{\underline p}$, hence Bertini's theorem implies that a general member is smooth away from $Z_{\underline p}$.
	
	It remains to analyze the singularities at the points of $Z_{\underline p}$. Fix one of them, say $p=[1:0:\cdots:0]$, and use affine coordinates $T_i=x_i/x_1$ for $i=2,\ldots,n$. A form in $H^0(\PP_k^{n-1},\mathcal I_{Z_{\underline p}}^2(d))$ has no constant or linear terms in these coordinates, and its quadratic part is an arbitrary quadratic form in $T_2,\ldots,T_n$, since the monomials $x_1^{d-2}x_ix_j$ belong to the linear system for all $2\le i,j\le n$. Therefore, for a general member, the quadratic part at $p$ is nondegenerate, so $p$ is an ordinary double point. The same argument applies at every point of $Z_{\underline p}$.
	
	Thus, on every fiber, the desired condition defines a nonempty open subset. Hence $\mathcal N_{n,d}$ is a nonempty open subset of $\mathcal L_{n,d}$. The irreducibility and dimension statement follow from Lemma~\ref{lem:linear-system-dim}.
\end{proof}

Let $\mathscr N_{n,d}\subseteq \PP(R_d)$ be the image of $\mathcal N_{n,d}$ under the projection $\pi:\mathcal N_{n,d}\to \PP(R_d)$. Thus $\mathscr N_{n,d}$ is the locus of reduced hypersurfaces of degree $d$ with exactly $n$ nodes in general linear position.

\begin{Proposition}\label{prop:nodal-locus-irred}
	The locus $\mathscr N_{n,d}$ is a nonempty irreducible constructible subset of $\PP(R_d)$ of dimension
	\[
	\dim \mathscr N_{n,d}=N_d-n.
	\]
\end{Proposition}

\begin{proof}
	Since $\mathcal N_{n,d}$ is irreducible, so is its image $\mathscr N_{n,d}$. If $[f]\in \mathscr N_{n,d}$, then its singular set consists of exactly $n$ points, so the fiber of $\pi$ over $[f]$ is finite: it consists precisely of the $n!$ orderings of the singular points of $V(f)$. Hence $\pi$ is generically finite onto its image, and therefore
	\[
	\dim \mathscr N_{n,d}=\dim \mathcal N_{n,d}=N_d-n.
	\]
\end{proof}

The nodal constructions above have immediate consequences for the geometry of the Veronese-Avoiding locus.

\begin{Theorem}\label{thm:nodal-stratum-in-VAH}
	For every $n\ge 3$ and $d\ge 3$, one has
	\[
	\mathscr N_{n,d}\subseteq \mathscr V_{n,d}^{\mathrm{red}}.
	\]
	In particular, $\mathscr V_{n,d}^{\mathrm{red}}$ is nonempty.
\end{Theorem}

\begin{proof}
	Every hypersurface in $\mathscr N_{n,d}$ has exactly $n$ nodes in general linear position. Hence it is Veronese-Avoiding by Theorem~\ref{thm:main-singular}. The nonemptiness follows from Proposition~\ref{prop:nodal-locus-irred}.
\end{proof}

\begin{Corollary}\label{cor:smooth-VAH-nonempty}
	For every $n\ge 3$ and $d\ge 3$, the smooth Veronese-Avoiding locus $\mathscr V_{n,d}^{\mathrm{sm}}$ is nonempty.
\end{Corollary}

\begin{proof}
	By Theorem~\ref{thm:nodal-stratum-in-VAH}, the set $\mathscr V_{n,d}^{\mathrm{red}}$ is nonempty. By~\cite[Lemma~4.5]{Alper-Isaev1}, the existence of a Veronese-Avoiding hypersurface of degree $d$ implies the existence of a smooth Veronese-Avoiding hypersurface of the same degree. Hence $\mathscr V_{n,d}^{\mathrm{sm}}\neq \emptyset$.
\end{proof}

\begin{Corollary}\label{cor:component-lower-bound}
	The variety $\mathscr V_{n,d}$ has an irreducible component of dimension at least
	\[
	N_d-n=\binom{n+d-1}{d}-1-n.
	\]
\end{Corollary}

\begin{proof}
	By Theorem~\ref{thm:nodal-stratum-in-VAH}, the irreducible set $\mathscr N_{n,d}$ is contained in $\mathscr V_{n,d}$. Hence some irreducible component of $\mathscr V_{n,d}$ contains $\mathscr N_{n,d}$, and therefore has dimension at least $\dim \mathscr N_{n,d}=N_d-n$.
\end{proof}

The next result describes the Veronese-Avoiding locus on the stratum of hypersurfaces with exactly $n$ singular points.

\begin{Theorem}\label{thm:classification-on-n-point-stratum}
	On the locus of reduced hypersurfaces of degree $d\ge 3$ with exactly $n$ isolated singular points, the Veronese-Avoiding condition is equivalent to having $n$ nodes in general linear position. Equivalently,
	\[
	\mathscr V_{n,d}\bigcap
	\{\, [f]\in \PP(R_d)\mid |\Sing(V(f))|=n \,\}
	=
	\mathscr N_{n,d}.
	\]
\end{Theorem}

\begin{proof}
	This is exactly Theorem~\ref{thm:main-singular}.
\end{proof}

\begin{Remark}
	Theorem~\ref{thm:classification-on-n-point-stratum} shows that the nodal stratum $\mathscr N_{n,d}$ is the natural singular part of the Veronese-Avoiding locus on the stratum of hypersurfaces with exactly $n$ singular points. In this range, the Veronese-Avoiding condition has a purely geometric characterization.
\end{Remark}

\begin{Remark}
	The preceding results determine a distinguished irreducible subvariety
	\[
	\mathscr N_{n,d}\subseteq \mathscr V_{n,d}
	\]
	of dimension $N_d-n$. It would be interesting to determine whether $\mathscr N_{n,d}$ is itself an irreducible component of $\mathscr V_{n,d}$, and more generally to describe the loci inside $\mathscr V_{n,d}$ corresponding to hypersurfaces with $r<n$ nodes imposing independent linear conditions. By Theorem~\ref{thm:nodal-r-points}, these loci are governed by the geometry of the rational map
	\[
	\phi_f:\PP^{\,n-r-1}\dashrightarrow \PP^{\,n-r-1},
	\]
	which suggests a finer stratification of $\mathscr V_{n,d}$ according to the base locus of $\phi_f$.
\end{Remark}
\section{A Lefschetz consequence of the Veronese-Avoiding condition}
\label{section:Lefschetz}

In this section we record a Lefschetz-type consequence of the
Veronese-Avoiding condition. Although the main result of this section follows
from \cite[Theorem~1.5 and Remark~1.7]{Alex-Giovanna-Abbas}, we include a
direct proof in the special setting of Milnor algebras. This proof makes clear
how the Veronese-avoidance condition enters.

Let $A=\bigoplus_{i\ge 0}A_i$ be a standard graded $k$-algebra. We say that
$A$ has the \emph{weak Lefschetz property in degree $q$} if, for a general
linear form $\ell\in A_1$, the multiplication map
\[
\times \ell:A_{q-1}\longrightarrow A_q
\]
has maximal rank. More generally, a multiplication map of the form
\[
\times \ell^s:A_i\longrightarrow A_{i+s}
\]
will be called a \emph{Lefschetz-type map}. When such a map has maximal rank
for a general linear form $\ell$, we say that $A$ satisfies the corresponding
Lefschetz-type property.

Let now $X=V(f)\subseteq \PP_k^{n-1}$ be a reduced hypersurface of degree
$d\ge 3$, and set
\[
T=n(d-2),\qquad m=T-1.
\]
Put $L_f:=(J_f)_m\subseteq R_m$. If $X$ is Veronese-Avoiding, then condition {\rm (I)} in Definition~\ref{VAH} gives $\dim_k R_m/L_f=n$. 
Moreover, condition {\rm (II)} says that
$\PP(L_f)\cap V_n^m=\emptyset$
or equivalently,
\[
\ell^m\notin L_f
\qquad\text{for every }0\neq \ell\in R_1.
\]

Let
\[
\pi:R_m\longrightarrow R_m/L_f
\]
be the natural quotient map. It induces a rational map
\[
\PP(R_m)\dashrightarrow \PP(R_m/L_f),\qquad [F]\longmapsto [\pi(F)].
\]
This map is not defined at the points $[F]\in \PP(R_m)$ for which
$\pi(F)=0$, that is, precisely at the points with $F\in L_f$. Thus its
indeterminacy locus is the projective linear subspace $\PP(L_f)$. After removing
this center, we obtain the linear projection
\[
\pi_{L_f}:\PP(R_m)\setminus \PP(L_f)\longrightarrow \PP(R_m/L_f),
\qquad [F]\longmapsto [F\bmod L_f].
\]

Recall that the Veronese variety $V_n^m\subseteq \PP(R_m)$ is the image of the
Veronese embedding
\[
v_n^m:\PP(R_1)\longrightarrow \PP(R_m),\qquad [\ell]\longmapsto [\ell^m].
\]
Since $X$ is Veronese-Avoiding, the center $\PP(L_f)$ is disjoint from
$V_n^m$. Hence every point $[\ell^m]$ of the Veronese variety lies in the domain
of the projection $\pi_{L_f}$. Therefore the composition
\[
\Phi_f:=\pi_{L_f}\circ v_n^m
\]
is a well-defined morphism
\[
\Phi_f:\PP(R_1)\longrightarrow \PP(R_m/L_f),
\qquad
[\ell]\longmapsto [\ell^m\bmod L_f].
\]
Since $\dim_k R_m/L_f=n$, the target is the projective space associated to an
$n$-dimensional vector space; hence
\[
\PP(R_m/L_f)\simeq \PP_k^{n-1}.
\]
Thus $\Phi_f$ is a morphism from $\PP(R_1)\simeq \PP_k^{n-1}$ to another
projective space of dimension $n-1$.

\begin{Theorem}\label{thm:VAH-degree-one-Lefschetz}
Let $X=V(f)\subseteq \PP_k^{n-1}$ be a Veronese-Avoiding hypersurface of degree
$d\ge 3$. Set $T=n(d-2)$. Then, for a
general linear form $\ell\in R_1$, the multiplication map
\[
\ell^{\,T-2}:(M_f)_1\longrightarrow (M_f)_{T-1}
\]
is an isomorphism.
\end{Theorem}

\begin{proof}
Set $m=T-1$ and $L_f=(J_f)_m$. Since $X$ is Veronese-Avoiding and $J_f$ is generated in degree $d-1\ge 2$, it follow that 
\[
\dim_k R_m/L_f=\dim_k(M_f)_m=\dim_k(M_f)_1=n.
\]
It is enough to prove that, for a general linear form $\ell\in R_1$, the map
\[
R_1\longrightarrow R_m/L_f,
\qquad
h\longmapsto h\ell^{m-1}\bmod L_f
\]
is injective.

Consider the morphism
\[
\Phi_f:\PP(R_1)\longrightarrow \PP(R_m/L_f),
\qquad
[\ell]\longmapsto [\ell^m\bmod L_f].
\]
This morphism is given by homogeneous forms of degree $m$ in the coordinates of
$\PP(R_1)$. Equivalently,
\[
\Phi_f^*\mathcal O_{\PP(R_m/L_f)}(1)\simeq \mathcal O_{\PP(R_1)}(m).
\]
Indeed, $\Phi_f$ is the composition of the $m$-th Veronese embedding with a linear projection, and the pullback of a hyperplane under the $m$-th Veronese
embedding is a hypersurface of degree $m$.

We claim that $\Phi_f$ cannot contract a positive-dimensional subvariety. It is enough to exclude the contraction of curves.  Indeed, if a positive-dimensional projective subvariety $Y\subseteq \PP(R_1)$ were contracted by $\Phi_f$ to a point, then $Y$ would contain an irreducible curve
$C\subseteq Y$, and this curve would also be contracted. Thus it suffices to
show that no irreducible curve is contracted. Suppose that an irreducible curve
$C\subseteq \PP(R_1)$ were contracted by $\Phi_f$ to a point. Then
\[
\deg\left(\Phi_f^*\mathcal O_{\PP(R_m/L_f)}(1)|_C\right)=0,
\]
because the pullback of a line bundle from a point has degree zero. On the other
hand, using the above identity, we get
\[
\Phi_f^*\mathcal O_{\PP(R_m/L_f)}(1)|_C
\simeq
\mathcal O_{\PP(R_1)}(m)|_C.
\]
Since $m>0$, this line bundle has positive degree on $C$, a contradiction. Therefore $\Phi_f$ does not contract any curve, and hence it does not contract
any positive-dimensional subvariety.

It follows that every fiber of $\Phi_f$ is zero-dimensional. Since $\PP(R_1)$ is projective, every zero-dimensional fiber is finite. Hence $\Phi_f$ has finite fibers. Finally, both the source and the target have dimension $n-1$. If the image of $\Phi_f$ had dimension strictly smaller than $n-1$, then the fiber
dimension theorem would imply that a general fiber has positive dimension, contradicting the finiteness of the fibers. Thus the image has dimension
$n-1$, and so $\Phi_f$ is dominant. Since $\Phi_f$ is projective and has finite
fibers, it is finite.

By the Generic Smoothness Theorem
\cite[Chapter~III, Corollary~10.7]{Hartshorne}, since
$\operatorname{char}(k)=0$, the finite dominant morphism
\[
\Phi_f:\PP(R_1)\longrightarrow \PP(R_m/L_f)
\]
is generically smooth. As $\Phi_f$ is finite, generic smoothness is equivalent
to generic \'etaleness. Hence, for a general point $[\ell]\in \PP(R_1)$, the
differential
\[
d\Phi_f|_{[\ell]}:T_{[\ell]}\PP(R_1)\longrightarrow
T_{\Phi_f([\ell])}\PP(R_m/L_f)
\]
is injective.

We now describe this differential explicitly. The tangent space to
$\PP(R_1)$ at $[\ell]$ is naturally identified with $R_1/k\ell$. Since
\[
\Phi_f([\ell])=[\ell^m\bmod L_f],
\]
the differential of $\Phi_f$ at $[\ell]$ is induced, up to the nonzero scalar
$m$, by the linear map
\[
R_1\longrightarrow R_m/L_f,\qquad
h\longmapsto h\ell^{m-1}\bmod L_f.
\]
We claim that the map
\[
R_1\longrightarrow R_m/L_f,\qquad
h\longmapsto h\ell^{m-1}\bmod L_f
\]
is injective for a general $\ell$. Suppose that
\[
h\ell^{m-1}\in L_f
\]
for some $h\in R_1$. Then the class of $h$ in $R_1/k\ell$ is killed by
$d\Phi_f|_{[\ell]}$. Since $d\Phi_f|_{[\ell]}$ is injective for general
$[\ell]$, it follows that $h\in k\ell$. Thus $h=c\ell$ for some $c\in k$.
Hence
\[
c\ell^m=h\ell^{m-1}\in L_f.
\]
But the Veronese-Avoiding condition gives $\ell^m\notin L_f$ for every
nonzero $\ell\in R_1$. Therefore $c=0$, and so $h=0$. This proves the desired
injectivity.

Since both $R_1$ and $R_m/L_f$ have dimension $n$, the above injective map is
an isomorphism for a general linear form $\ell$. Recalling that $m=T-1$, this
is exactly the multiplication map
\[
\ell^{T-2}:(M_f)_1\longrightarrow (M_f)_{T-1}.
\]
The proof is complete.
\end{proof}

\begin{Remark}
Theorem~\ref{thm:VAH-degree-one-Lefschetz} is a Lefschetz-type consequence of
the definition of Veronese-Avoiding hypersurfaces. It does not assert that the
Milnor algebra has the full strong Lefschetz property. Rather, it says that the
top Jacobian piece defining the Veronese-Avoiding condition forces the
distinguished multiplication map
\[
\ell^{T-2}:(M_f)_1\longrightarrow (M_f)_{T-1}
\]
to be an isomorphism for a general linear form $\ell$.
\end{Remark}

We close with two consequences showing how the Lefschetz-type isomorphism above
specializes on the singular loci classified in Section~\ref{sectionVAH}.
\begin{Remark}
Combining Theorem~\ref{thm:VAH-degree-one-Lefschetz} with the singular
classification results obtained above, we get the following consequences.
\begin{enumerate}
\item Let $X=V(f)\subseteq \PP_k^{n-1}$ be a reduced hypersurface of degree
$d\ge 3$ with exactly $n$ isolated singular points. By
Theorem~\ref{thm:main-singular}, $X$ is Veronese-Avoiding if and only if
$\Sing(X)$ consists of $n$ nodes in general linear position. Hence, in this
case, Theorem~\ref{thm:VAH-degree-one-Lefschetz} implies that, for a general
linear form $\ell$, the map
\[
\ell^{\,T-2}:(M_f)_1\longrightarrow (M_f)_{T-1}
\]
is an isomorphism.

\item In the case of reduced singular plane cubics, this Lefschetz-type
statement becomes an ordinary weak Lefschetz statement. Indeed, if
$X=V(f)\subseteq \PP_k^2$ is a reduced singular cubic curve, then
$n=d=3$, so $T=3$ and $T-2=1$. By
Proposition~\ref{prop:cubic-plane-classification}, $X$ is Veronese-Avoiding if
and only if $X$ is nodal. Therefore, for every reduced nodal cubic curve, and
for a general linear form $\ell$, the map
\[
\ell:(M_f)_1\longrightarrow (M_f)_2
\]
is an isomorphism. In particular, $M_f$ has the weak Lefschetz property in
degree $2$.
\end{enumerate}
\end{Remark}

\end{document}